%% file: root.tex
\documentclass[submission,copyright,creativecommons,hyphens,hidelinks,fleqn]{eptcs}
\providecommand{\event}{ACT 2021} 
\usepackage{underscore}           

\usepackage{sensible} 

\usepackage{mathtools} 
\usepackage{amssymb} 
\usepackage{amsthm} 
\makeatletter
\def\@endtheorem{\endtrivlist}
\makeatother
\usepackage{thmtools}
\usepackage{bbold} 
\usepackage{braket} 

\usepackage{enumitem} 
\usepackage{etoolbox}
\usepackage{xparse}

\usepackage{sensible-suffix} 
\usepackage{crossreftools}
\usepackage{custommacros} 

\title{Constructing Initial Algebras Using Inflationary Iteration}
\author{Andrew M. Pitts \qquad\qquad S. C. Steenkamp\thanks{Supported
    by UK EPSRC PhD studentship 2119809.} 
\institute{Department of Computer Science and Technology\\
University of Cambridge, UK}
\email{andrew.pitts@cl.cam.ac.uk \qquad s.c.steenkamp@cl.cam.ac.uk}
}
\def\titlerunning{Constructing Initial Algebras Using Inflationary Iteration}
\def\authorrunning{A. M. Pitts \& S. C. Steenkamp}

\begin{document}
\maketitle

\begin{abstract}
  An old theorem of Ad\'amek constructs initial algebras for
  sufficiently cocontinuous endofunctors via transfinite iteration
  over ordinals in classical set theory. We prove a new version that
  works in constructive logic, using ``inflationary'' iteration over a
  notion of \emph{size} that abstracts from limit ordinals just their
  transitive, directed and well-founded properties. Borrowing from
  Taylor's constructive treatment of ordinals, we show that sizes
  exist with upper bounds for any given signature of indexes. From
  this it follows that there is a rich class of endofunctors to which
  the new theorem applies, provided one admits a weak form of choice
  (WISC) due to Streicher, Moerdijk, van~den~Berg and Palmgren, and
  which is known to hold in the internal constructive
  logic of many kinds of topos.
\end{abstract}

\input{content}

\nocite{*}

\end{document}

%% file: content.tex
\section{Introduction}

Initial algebras for endofunctors are a simple category-theoretic
concept that has proved very useful in logic and computer
science. Recall that an \emph{initial algebra} $(\mu F,\iota)$ for an
endofunctor $F:\C\fun\C$ on a category $\C$ is a morphism
$\iota:F(\mu F)\fun \mu F$ in $\C$ with the property that for any
morphism $a:F(A)\fun A$, there is a unique $\lift{a}:\mu F\fun A$ that
is an $F$-algebra morphism, that is, satisfies
$\lift{a}\comp\iota = a\comp F(\lift{a})$. In functional programming,
$\lift{a}$ is sometimes called the \emph{catamorphism} associated with
the algebra $(A,a)$~\cite{MeijerE:funpbl}. By varying the choice of
$\C$ and $F$, such initial algebras give semantics for various kinds
of inductive (or dually, coinductive) structures and, via their
catamorphisms, associated (co)recursion schemes. We refer the reader
to the draft book by Ad{\'a}mek, Milius, and Moss~\cite{ammbook} for an
account of this within classical logic.

Here we make a contribution to the existence of initial algebras
within \emph{constructive} logics. Our reason for seeking a
constructive treatment is not philosophical, nor motivated by the
computational insights that a constructive approach can bring,
important though both those thing are. Rather, we are interested in
the semantics of dependent type theories with inductive constructions,
such as types that are inductive~\cite{Martin-LoefP:inttt},
inductive-recursive~\cite{DybjerP:genfsi},
inductive-inductive~\cite{ForsbergFN:indid}, quotient
(inductive-)inductive~\cite{AltenkirchT:quoiit,KaposiA:lariqi} and
more generally higher-inductive~\cite{HoTT}. Toposes are often used
when constructing models of such type theories and sometimes the
easiest way of doing so is to use their ``internal
logic''~\cite[Part~D]{JohnstonePT:skeett} to express the
constructions; see~\cite{PittsAM:aximct,PittsAM:intumh}, for
example. Although there are different candidates for what is the
internal logic of toposes, in general they are not classical. So we
are led to ask for what categories $\C$ and functors $F:\C\fun\C$ that
are describable in such an internal logic it is the case that an
initial $F$-algebra can be constructed.

We pursue this question by developing a constructive version of
Ad\'amek's classical theorem about existence of initial algebras via
transfinite iteration over ordinals~\cite{AdamekJ:freaar} (we discuss
a different constructive approach~\cite{AdamekJ:iniatw} in
\cref{sec:relw}). Recall, or see Ad\'amek
et~al.~\cite[section~6.1]{ammbook} for example, that if $F:\C\fun\C$
is an endofunctor on a category $\C$ with all small colimits (colimits
of small chains are enough), then we get a large chain in $\C$,
$(F^\alpha0)_{\alpha\in\Ord}$ indexed by the totally ordered class of
ordinals $\Ord$, defined by recursion over the ordinals:
\begin{equation}
  \label{eq:ordinal-iteration}
  F^\alpha0 =
  \begin{cases}
    0 \quad\text{(initial object in $\C$)}
    &\text{if $\alpha = 0$, the ordinal zero}\\
    F(F^\beta 0)
    &\text{if $\alpha=\beta^{+}$ is a successor ordinal}\\
    \colim_{\beta<\lambda}F^\beta 0
    &\text{if $\alpha=\lambda$ is a limit ordinal}
  \end{cases}
\end{equation}
The links in the chain are $\C$-morphisms $i_\alpha:F^\alpha 0\fun
F^{\alpha^{+}} 0$ also defined by ordinal recursion:
\begin{equation}
  \label{eq:ordinal-link}
  i_\alpha =
  \begin{cases}
    \text{unique morphism given by initiality of $0$}
    &\text{if $\alpha = 0$, the ordinal zero}\\
    F(i_\beta)
    &\text{if $\alpha=\beta^{+}$ is a successor ordinal}\\
    \text{induced by the universal property of colimits}
    &\text{if $\alpha=\lambda$ is a limit ordinal}
  \end{cases}
\end{equation}

\begin{theorem}[\classical{} (Ad\'amek~\cite{AdamekJ:freaar})]
  If $i_\alpha$ is an isomorphism for some $\alpha\in\Ord$, then
  $(F^\alpha 0, i_\alpha^{-1})$ is an initial algebra for
  $F:\C\fun\C$. So in particular, if $F$ preserves colimits of shape
  $\lambda$ for some limit ordinal $\lambda$, then (by the definition of
  ``preserves colimits'') $i_\lambda$ is an isomorphism and
  $F^\lambda 0$ is an initial $F$-algebra.
\end{theorem}
This theorem is labelled \classical{} because its proof uses classical
logic: the properties of ordinal numbers that it relies upon require
the Law of Excluded Middle ($\forall p.\; p \disj \neg p$). In
\cref{sec:infl-iter} we show that by replacing the use of
ordinals with a weaker notion of ``size'' and modifying the way $F$ is
iterated, one can obtain a constructive version of Ad\'amek's theorem
(see \cref{thm:infl-iter}).

Not only the proof, but also the application of Ad\'amek's theorem can
require classical logic: the Axiom of Choice \ac{} is often invoked to
find a suitably large limit ordinal $\lambda$ for which a particular
functor of interest preserves $\lambda$-colimits. Such uses of \ac{}
are not always necessary. In particular, existence of initial algebras
for polynominal functors
$F_{A,B}(\_)= \sum_{a\in A} (\_)^{B(a)}:\Sets\fun\Sets$ (where
$A\in\Sets$ and $B\in\Sets^A$) can be proved constructively;
see~\cite[Proposition~3.6]{MoerdijkI:weltc}. These initial algebras
are the categorical analogue of
W-types~\cite{abbott2005containers,GambinoN:welftd} and we will make
use of the fact that they exist in toposes with natural number object
in what follows. However, for non-polynomial functors, especially ones
whose specification involves both exponentiation by infinite sets and
taking quotients by equivalence relations (such as~\cref{exa:sym-cont}
below), it is not immediately clear that \ac{} can be avoided. In
fact, we show in \cref{sec:sized-endof} that a much weaker choice
principle than \ac, the ``Weakly Initial Sets of Covers'' \wisc{}
axiom~\cite{StreicherT:reamc,MoerdijkI:weltc,vandenberg2014axiom}, is
enough to ensure that our constructive version of Ad\'amek's theorem
applies to a rich class of endofunctors. \wisc{} has been called
``constructively acceptable'' because it is valid in a wide range of
elementary toposes~\cite{vandenberg2014axiom}. In particular it holds
in presheaf and realizability toposes that have been used to construct
models of dependent type theory that mix quotients and inductive
constructions, which, as we mentioned above, motivates our pursuit of
a constructive treatment of initial algebras.

\section{Constructive meta-theory}

The results in this paper are presented in the usual informal language
of mathematics, but only making use of intuitionistically valid logical
principles (and, to obtain the results of \cref{sec:sized-endof},
extended by the WISC axiom). In particular we avoid use of the Law of
Excluded Middle, or more generally the Axiom of Choice.

More specifically, our results can be soundly interpreted in any
elementary topos with natural number object and
universes~\cite{StreicherT:unit} (satisfying \wisc, for the last part
of the paper). Thus when we refer to the category $\Sets$ of small
sets and functions, we mean the generalised elements of some such
universe, which we always assume contains the subobject classifier. In
fact, in order to interpret quantification over such small sets in a
straightforward way, we tacitly assume there is a countable nested
sequence of such universes, $\Sets=\Sets_0 \in \Sets_1\in\cdots$. A
suitable version of Martin-L\"of's Extensional Type
Theory~\cite{Martin-LoefP:inttt} extended with an impredicative
universe of propositions can be used as the internal language of such
toposes.

In fact the use of impredicative quantification is not necessary: we
have developed a formalisation of the results of this paper using the
Agda proof assistant~\cite{Agda261}, which can provide a dependent
type theory with a predicative universe of (proof irrelevant)
propositions and convenient mechanisms (such as pattern-matching) for
using inductively defined types.  We then have to postulate as axioms
some things which are derivable in the logic of toposes, namely axioms
for propositional extensionality, quotient sets and unique choice (and
\wisc, when we need it). Our Agda development is available at
\cite{agdacode}.

\section{Size-indexed inflationary iteration}
\label{sec:infl-iter}

Throughout this section we fix a large, locally small
category\footnote{The collection of objects is in $\Sets_1$ and the
  collection of morphisms between any pair of objects is in $\Sets$.}
$\C$ and an endofunctor $F:\C\fun\C$. We will consider sequences of
objects in $\C$ built up by iterating $F$ while taking certain
colimits. For simplicity we assume that $\C$ is cocomplete, that is,
has colimits of all small diagrams.\footnote{This means that we are
  given a function assigning a choice of colimit for each small
  diagram, since we work in a constructive setting and in particular
  have to avoid the use of the Axiom of Choice.}

From a constructive point of view, the problem with the sequence
\eqref{eq:ordinal-iteration} is that it makes use of ordinals, which
rely on the Law of Excluded Middle \lem{} for their good properties; in
particular, the definition in \eqref{eq:ordinal-iteration} is by cases
according to whether an ordinal is zero, or a successor, \emph{or
  not}.  In the case that $\C$ is a complete partially ordered set
(with joins denoted by $\bigvee$), Abel and
Pientka~\cite[section~4.5]{AbelA:wellfrc} point out that one can avoid
this case distinction, while still achieving within constructive logic
the same result in the (co)limit, by instead taking the approach of
Sprenger and Dam~\cite{SprengerC:strirc} and using what they term an
\emph{inflationary iteration}:
\begin{equation}
  \label{eq:inflationary-poset}
  \mu_i F = \textstyle\bigvee_{j<i}F(\mu_j F)
\end{equation}
We only need $i$ to range over the elements of a set equipped with a
binary relation $<$ that is well-founded for this definition to make
sense. Here we generalise from complete posets to cocomplete categories,
replacing joins by colimits. \cref{def:size} sums up what we need of the
indexes $i$ and the relation $<$ between them in order to ensure that
the inflationary sequence can be defined and yields an initial algebra
for $F$ if it becomes stationary up to isomorphism.

\begin{definition}
  \label{def:semi-cat}
  Recall that a \emph{semi-category} is like a category, but lacks
  identity morphisms. A semi-category is \emph{thin} if there is at
  most one morphism between any pair of objects. Thus a small thin
  semi-category is the same thing as a set $\kappa$ (the set of
  objects) equipped with a transitive relation
  $\_<\_ \subseteq \kappa\times\kappa$ (the existence-of-a-morphism
  relation). Given such a $(\kappa,{<})$, a \emph{diagram}
  $D:\kappa\fun\C$ in a category $\C$ is by definition a semi-functor
  from $\kappa$ to $\C$: thus $D$ maps each $i\in\kappa$ to a
  $\C$-object $D_i$, each pair $(j,i)$ with $j<i$ to a $\C$-morphism
  $D_{j,i}: D_j\fun D_i$, and these morphisms satisfy
  $D_{j,i}\comp D_{k,j} = D_{k,i}$ for all $k<j<i$ in $\kappa$.
\end{definition}

\begin{definition}
  \label{def:size}
  A \emph{size} is a small thin
  semi-category $(\kappa,{<})$ that is
  \begin{itemize}

  \item \emph{directed}: every finite subset of $\kappa$ has an upper
    bound with respect to $<$; specifically, we assume we are given a
    distinguished element $\sizezero\in\kappa$ and a binary operation
    $\_\ub\_:\kappa\times\kappa\fun\kappa$ satisfying
    $\forall i,j\in\kappa.\; i< {i \ub j}
    \;\conj\; j < {i\ub j}$

  \item \emph{well-founded}: for all $K\subseteq \kappa$, if
    $\forall i\in\kappa.(\forall j< i.\; j \in K)
    \imp i\in K$, then $K = \kappa$.
    \qedhere
  \end{itemize}
  Note that the directedness property in particular gives a successor
  operation $\sizesucc : \kappa\fun\kappa$ on the elements of a size,
  defined by $\sizesucc i \defeq i\ub i$ and satisfying
  $\forall i\in\kappa.\; i < \sizesucc i$. (We do not need a successor
  that also preserves $<$, although the sizes constructed in the next
  section have one that does so.)
\end{definition}

\begin{example}
  \label{exa:nat}
  In the next section we will define a rich class of sizes derived
  from algebraic signatures (see
  \cref{prop:plump-size}). For now, we note that the
  natural numbers $\Nat$ with their usual strict order is a
  size.\footnote{$\Nat$ will be the smallest size once one has developed a
    comparison relation between sizes. To do that one probably has to
    restrict to sizes that are \emph{extensional}, that is, satisfy
    $\forall i,j\in\kappa.\;\{k\in\kappa\mid k<i\} = \{k\in\kappa\mid
    k<j\} \imp i = j$.  However, we have no need of that property for
    the results in this paper.} In classical logic, an ordinal is a
  size iff its usual strict total order is directed, which happens iff
  it is a limit ordinal.
\end{example}

\begin{remark}
  \label{rem:wf-rec}
  Since we are working constructively, the well-foundedness property
  of a size is stated in a suitably positive form; classically, it is
  equivalent to the non-existence of infinite descending chains for
  $<$. Well-foundedness of $<$ allows one to define size-indexed
  families by \emph{well-founded
    recursion}~\cite[section~6.3]{TaylorP:prafm}: given a size
  $\kappa$ and a $\kappa$-indexed family of sets
  $(A_i)_{i\in\kappa}$, from each family of functions
  $(f_i:(\prod_{j<i}A_j)\fun
  A_i)_{i\in\kappa}$ we get a family of elements
  $(a_i\in A_i)_{i\in\kappa}$, uniquely defined by the
  requirement
  $\forall i\in\kappa.\; a_i=
  f_i((a_j)_{j<i})$.
\end{remark}
Given a size $\kappa$, for each element $i\in\kappa$ we get a small thin
semi-category\footnote{Well-foundedness is preserved, but directedness
is not, so \(\subsize(i)\) is not necessarily a size.} $\subsize(i)$
whose vertices are the elements $j\in\kappa$ with $j<i$ and whose
morphisms are the instances of the $<$ relation. Thus a diagram
$D:\subsize(i)\fun\C$ maps each $j<i$ to a $\C$-object $D_j$ and each
pair $(k,j)$ with $k<j<i$ to a $\C$-morphism $D_{k,j}: D_k\fun D_j$,
satisfying $D_{k,j}\comp D_{l,k} = D_{l,j}$ for all $l<k<j<i$. We write
\begin{equation}
  \label{eq:colim}
  (\inc^D_j:D_j \fun \colim_{j<i}D_j)_{j<i}
\end{equation}
for the colimit of this diagram (recall that we are assuming $\C$ is
cocomplete). Thus for all $k<j<i$ it is the case that
$\inc^D_k = \inc^D_j \comp D_{k,j}$; and given any
cocone in $\C$
\[
(f_j : D_j\fun X)_{j<i} \qquad \forall
k<j<i.\; f_k = f_j\comp D_{k,j}
\]
there is a unique $\C$-morphism $\hat{f}:\colim_{j<i}D_j
\fun X$ satisfying $\forall j<i.\; \hat{f}\comp\inc^D_j =
f_j$.

Since $<$ is transitive, if $j<i$ in $\kappa$, then
$\subsize(j)$ is a sub-semi-category of $\subsize(i)$ and each diagram
$D:\subsize(i)\fun\C$ restricts to a diagram
$D|_j:\subsize(j)\fun\C$. We write
\begin{equation}
  \label{eq:inc-hat}
  \c^D_{j,i} : \colim_{k<j}D_k \fun
  \colim_{k<i}D_k
\end{equation}
for the unique $\C$-morphism satisfying $\forall k<j<i.\;
\c^D_{j,i}\comp \inc^{D|_j}_k = \inc^D_k$.

\begin{definition}
  \label{def:infl-iter}
  Let $\kappa$ be a size. Given an endofunctor $F:\C\fun\C$ on a
  cocomplete category $\C$, a diagram $D:\kappa\fun\C$ is an
  \emph{inflationary iteration of $F$ over $\kappa$} if for all
  $i\in\kappa$
 \begin{equation*}
   \label{eq:infl-iter}
   D_i = \colim_{j<i} F(D_j) \;\conj\;
   \forall j<i.\; D_{j,i} = \c^{F\comp D}_{j,i}
   \qedhere
 \end{equation*}
\end{definition}

\begin{lemma}
  \label{lem:infl-iter}
  Given an endofunctor $F:\C\fun\C$ on a cocomplete category $\C$, for
  each size $\kappa$ an inflationary iteration of $F$ over $\kappa$
  exists (and is unique).
\end{lemma}
\begin{proof}
  Given $i\in\kappa$, say that a diagram $D:\subsize(i)\fun\C$ is an
  inflationary iteration of $F$ \emph{up to} $i$ if for all $j<i$,
  $D_j = \colim_{k<j}F(D_k)$ and
  $\forall k<j.\; D_{k,j}= \c^{F\comp D}_{k,j}$.  Note that given such
  a diagram, for any $j<i$ we have that $D|_j:\subsize(j)\fun\C$ is
  an inflationary iteration of $F$ up to $j$. Using well-founded
  induction for $<$, one can prove that
  \begin{equation}
    \label{eq:upto-unique}
    \forall i\in\kappa,\;\text{any two inflationary iterations of $F$
      up to $i$ are
      equal}
  \end{equation}
  Then one can use well-founded recursion for $<$
  (\cref{rem:wf-rec}) to define for each $i\in\kappa$ an
  inflationary iteration of $F$ up to $i$,
  $D^{(i)}:\subsize(i)\fun\C$. If $j<i$, then
  $D^{(j)}$ and $D^{(i)}|_j$ are both inflationary
  iterations of $F$ up to $j$ and so are equal by
  \eqref{eq:upto-unique}. From this it follows that
  \[
    (D_i \defeq \colim_{j<i}
    F(D^{(i)}_j))_{i\in\kappa} \;\conj\;
    (D_{j,i} \defeq \c^{F\comp
      D^{(i)}}_{j,i})_{j,i\in\kappa\mid
      j<i}
  \]
  defines an inflationary iteration of $F$. (Furthermore, since any
  such restricts to an up-to-$i$ inflationary iteration, uniqueness
  follows from \eqref{eq:upto-unique}.)
\end{proof}

\begin{remark}
  \label{rem:infl-iter-props}
  We record some simple properties of inflationary iteration that we
  need in the proof of the theorem below. Let $D:\kappa\fun\C$ be the
  inflationary iteration of $F:\C\fun\C$ over $\kappa$. Note that for
  all $j<i$ in $\kappa$, the components of the colimit cocone
  $\inc^{F\comp D|_i}_j: F(D_j) \fun
  \colim_{j<i}F(D_j)$ are morphisms
  $\iota_{j,i}: F(D_j) \fun D_i$ satisfying
  \begin{equation}
    \label{eq:iota-props}
    \forall k<j<i.\; D_{j,i}\comp
    \iota_{k,j} = \iota_{k,i} =
    \iota_{j,i}\comp F(D_{k,j})
  \end{equation}
  The first equation follows from the fact that $D_{j,i} =
  \c^{F\comp D}_{j,i}$ and the second from the definition of
  $\iota_{j,i}$ as a component of a cocone. Since that cocone
  is colimiting, one also has for all $i\in\kappa$ and all
  $\C$-morphisms $f,g:D_i\fun X$ that
  \begin{equation}
    \label{eq:iota-ext}
    (\forall j<i.\; f\comp\iota_{j,i} =
    g\comp\iota_{j,i}) \imp f = g 
  \end{equation}
\end{remark}
The proof of \cref{lem:infl-iter} only used the transitive and
well-founded properties of the relation $<$ on a size $\kappa$,
whereas the following theorem needs its directedness property as well.

\begin{theorem}[\textbf{(Initial algebras via inflationary iteration)}]
  \label{thm:infl-iter}
  Suppose $\C$ is a cocomplete category, $F:\C\fun\C$ is an
  endofunctor and there is a size $\kappa$ such that $F$ preserves
  colimits of diagrams $\kappa\fun\C$. Then $F$ has an initial algebra
  whose underlying $\C$-object is the colimit
  $\mu F = \colim_{i\in\kappa}\mu_i F$ of the inflationary iteration
  $\mu_i F$ (\cref{def:infl-iter}) of $F$ over $\kappa$.
\end{theorem}
\begin{proof}
  By \cref{lem:infl-iter} there is an inflationary iteration of
  $F:\C\fun\C$ over $\kappa$; call it $D:\kappa\fun\C$ and define
  $\mu F \defeq \colim_{i\in\kappa}D_i$. For each $i\in\kappa$, as in
  \cref{def:size} we have $\sizesucc i\in\kappa$ with
  $i < \sizesucc i$ and hence a $\C$-morphism
  \[
    \iota_i \defeq
    \left(F(D_i)\xrightarrow{\iota_{i,\sizesucc i}}
      D_{\sizesucc i} \xrightarrow{\inc^D_{\sizesucc i}}
      \colim_{i\in\kappa}D_i = \mu F\right)
  \]
  By \eqref{eq:iota-props}, $(\iota_i)_{i\in\kappa}$ is a
  cocone under the diagram $F\comp D : \kappa\fun\C$ and so induces
  $\hat{\iota}: \colim_{i\in\kappa}F(D_i) \fun \mu F$.  Then
  since $F$ preserves the colimit of $D$, we get a morphism
  \begin{equation}
    \label{eq:iota-def}
    \iota \defeq \left(F(\mu F) = F(\colim_{i\in\kappa}D_i) \iso
     \colim_{i\in\kappa} F(D_i) \xrightarrow{\hat{\iota}}
     \mu F\right)
  \end{equation}
  Therefore $\mu F$ has the structure of an $F$-algebra. To see that
  it is initial, suppose we are given $a:F(A)\fun A$. We have to show
  that there is a unique $F$-algebra morphism
  $(\mu F, \iota)\fun (A,a)$.

  If $h:\mu F \fun A$ is such an algebra morphism, that is
  $h\comp\iota = a\comp F(h)$, then by definition of $\iota$ in
  \eqref{eq:iota-def} it follows that the associated cocone
  $(h_i \defeq h \comp \inc^D_i:D_i\fun
  A)_{i\in\kappa}$ satisfies
  $h_{\sizesucc i}\comp \iota_{i,\sizesucc i}= a\comp
  F(h_i)$. From this, using the directedness property of sizes,
  we get
  \begin{equation}
    \label{eq:upto-morphism}
    \forall i\in\kappa.\forall j<i.\; 
    h_i\comp \iota_{j,i} = a\comp F(h_i\circ
    D_{j,i})
  \end{equation}
  So if $h$ and $h'$ are both $F$-algebra morphisms
  $(\mu F, \iota)\fun (A,a)$, one can prove by well-founded induction
  for $<$, using \eqref{eq:iota-ext} and \eqref{eq:upto-morphism},
  that
  $\forall i\in\kappa.\; h\comp \inc^D_i =
  h'\comp\inc^D_i$ and hence that $h=h'$.

  So it just remains to prove that there is such an $h$. It suffices
  to construct a cocone $(h_i:D_i\fun A)_{i\in\kappa}$ satisfying
  \eqref{eq:upto-morphism} and then take $h$ to be the morphism given
  by the universal property of the colimit; for then we have
  $\forall i\in\kappa.\; h_{\sizesucc i}\comp \iota_{i,\sizesucc i} =
  a\comp F(h_i)$ and hence $h\comp\iota = a\comp F(h)$, as required.

  For each $i\in\kappa$, say that a morphism $h':D_i\fun A$ is an
  \emph{up-to-$i$ algebra morphism} if
  $\forall j<i.\;h'\comp\iota_{j,i} = a\comp F(h'\comp D_{j,i})$
  (\emph{cf}.~\eqref{eq:upto-morphism}). Given such a morphism, then
  for any $j<i$, $h'\comp D_{j,i}:D_j\fun A$ is an up-to-$j$ algebra
  morphism. From this it follows by well-founded induction for $<$
  that any two up-to-$i$ algebra morphisms are equal. A well-founded
  recursion for $<$ allows one to construct an up-to-$i$ algebra
  morphism $h_i:D_i\fun A$ for each $i\in\kappa$; and the uniqueness
  of up-to algebra morphisms implies that $h_j=h_i\comp D_{j,i}$ when
  $j<i$. Thus $(h_i)_{i\in\kappa}$ is the required cocone satisfying
  \eqref{eq:upto-morphism}.
\end{proof}

\begin{corollary}
  \label{cor:infl-iter}
  With the same assumptions on $\C$, $F$ and $\kappa$ as in
  \cref{thm:infl-iter}, then free $F$-algebras exist, that is,
  the forgetful functor from the category of $F$-algebras to $\C$ has
  a left adjoint.
\end{corollary}
\begin{proof}
  The free $F$-algebra on an object $X\in\C$ is the same thing as an
  initial algebra for the endofunctor $F(\_)+X$. So by the theorem, it
  suffices to check that $F(\_)+X$ preserves colimits of diagrams
  $\kappa\fun\C$. It does so because $F$ does by assumption and
  because $\kappa$ is directed
  (cf.~\cref{prop:sized-closure}\eqref{item:4} below).
\end{proof}

\section{Initial algebras for sized endofunctors}
\label{sec:sized-endof}

In classical set theory with the Axiom of Choice, given a set of
operation symbols $A\in\Sets$ with associated arities $B\in\Sets^A$,
the associated polynomial endofunctor $X\mapsto \sum_{a\in A}X^{B(a)}$
on $\Sets$ preserves $\lambda$-colimits when the ordinal $\lambda$ is
large enough; specifically it does so if for all $a\in A$, $\lambda$
has upper bounds (with respect to the strict total order given by
membership) for all $B(a)$-indexed families of ordinals less than
$\lambda$. We will see that this notion of ``large enough'' is also
the right one for sizes in our constructive setting.

\begin{definition}
  \label{def:filtered}
  A \emph{signature} (also known as a
  \emph{container}~\cite{abbott2005containers,GambinoN:welftd}) is
  specified by a set $A\in\Sets$ and an $A$-indexed family of sets
  $B\in\Sets^A$. We write $\Sig \in \Sets_1$ for the large set of all
  such signatures.  Given $\Sigma=(A,B)\in\Sig$, we say that a size
  $(\kappa,{<})$ is \emph{$\Sigma$-filtered} if for all $a\in A$ and
  every function $f:B(a) \fun \kappa$, there exists $i\in\kappa$ with
  $\forall x\in B(a).\; f(x) < i$.
\end{definition}
We can deduce the existence of $\Sigma$-filtered sizes by abstracting
from the constructive analysis of Conway's surreal numbers by
Shulman~\cite{ShulmanM:surcpo}, which in turn is inspired by Taylor's
constructive notion of ``plump'' ordinal~\cite{TaylorP:intso}.  For
each $\Sigma=(A,B)\in\Sig$, let $W_{\Sigma}$ be the initial algebra
for the associated polynominal endofunctor $F_{A,B}:\Sets\fun\Sets$,
$F_{A,B}(X) = \sum_{a\in A}X^{B(a)}$. Thus $W_\Sigma$ is an example of
a W-type~\cite[Chapter~15]{NordstromB:progmlt}. The function
$\Sigma \mapsto W_{\Sigma}$ exists in our constructive setting,
because W-types can be constructed in elementary toposes with natural
number objects~\cite[Proposition~3.6]{MoerdijkI:weltc}; one can take
the elements of $W_\Sigma$ to be well-founded trees representing the
algebraic terms inductively generated by the signature $\Sigma$. Each
such term $t$ is uniquely of the form $\sup_{a} f$ where $\sup_{a}$ is
the $B(a)$-arity operation symbol named by $a\in A$ and, inductively,
$f=(t_x)_{x\in B(a)}$ is a $B(a)$-tuple of well-founded algebraic
terms over $\Sigma$.  The \emph{plump} ordering on $W_\Sigma$ is given
by the least relations ${\_<\_}\subseteq W_\Sigma\times W_\Sigma$ and
${\_\leq\_}\subseteq W_\Sigma\times W_\Sigma$ satisfying for all
$a\in A$, $f:B(a)\fun W_\Sigma$ and $t\in W_\Sigma$
\begin{equation}
  (\forall x\in B(a).\; f(x) < t) \imp \textstyle\sup_a f \leq
  t\quad\text{and}\quad
  (\exists x\in B(a).\; t \leq f(x)) \imp t < \textstyle\sup_a f
  \label{eq:plump}
\end{equation}
As noted in~\cite[Example~5.4]{PittsAM:quoitq}, $<$ is transitive and
well-founded, and $\leq$ is a preorder (reflexive and transitive). In
particular, since $\leq$ is reflexive, from \eqref{eq:plump} we deduce
that $\forall x\in B(a).\; f(x) < \sup_a f$, in other words for each
arity set $B(a)$ in the signature, any function $f:B(a)\fun W_\Sigma$
is bounded above in the $<$ relation by $\sup_a f$. This allows us to
construct $\Sigma$-filtered sizes:

\begin{proposition}
  \label{prop:plump-size}
  There is a $\Sigma$-filtered size $(\kappa_\Sigma , {<})$ for every
  signature $\Sigma$.
\end{proposition}
\begin{proof} Given a signature $\Sigma=(A,B)$, we extend it to a
  signature $(A',B')$ by adding fresh nullary and binary operation
  symbols. Thus $A' \defeq A\uplus\{n,b\}$ and $B'\in\Sets^{A'}$
  satisfies $B'(a) \defeq B(a)$ for $a\in A$, $B'(n)\defeq \emptyset$
  and $B'(b)\defeq \{0,1\}$. Let set $\kappa_\Sigma$ be the W-type
  $W_{(A',B')}$ and let $<$ be the plump order given by
  \eqref{eq:plump}. As noted above, $<$ is transitive and well-founded
  and has upper bounds for any arity-indexed family and hence in
  particular it is $\Sigma$-filtered. It just remains to see that it
  is directed (\cref{def:size}). Since $A'$ contains the
  nullary operation symbol $n$, $\kappa_\Sigma$ contains
  $\sizezero\defeq \sup_n\emptyset$; and given $i,j\in\kappa_\Sigma$,
  letting $f:B'(b)=\{0,1\}\fun \kappa_\Sigma$ map $0$ to $i$ and $1$
  to $j$, then $i\ub j \defeq \sup_b f$ is an upper bound for $i$ and
  $j$ with respect to $<$.
\end{proof}

\begin{definition}
  \label{def:sized}
  Given a signature $\Sigma\in\Sig$, a functor $F:\C\fun\D$ between
  cocomplete categories is \emph{$\Sigma$-sized} if it preserves
  colimits of all diagrams $\kappa\fun\C$ for any $\Sigma$-filtered
  size $\kappa$. A functor is \emph{sized} it there exists a signature
  $\Sigma$ for which it is $\Sigma$-sized.
\end{definition}

\begin{theorem}[\textbf{(Sized endofunctors have initial algebras)}]
  \label{thm:sized-endof}
  Assuming $\C$ is a cocomplete category, if $F:\C\fun\C$ is sized,
  then there exists an initial algebra for $F$. More precisely, there
  is a function assigning to each signature $\Sigma$ and each
  $\Sigma$-sized endofunctor $F$ an initial algebra for $F$.
\end{theorem}
\begin{proof}
  If $F$ is $\Sigma$-sized for some $\Sigma\in\Sig$, then $F$
  preserves colimits of diagrams for the $\Sigma$-filtered size
  $\kappa_\Sigma$ constructed in the proof of
  \cref{prop:plump-size}. Hence by \cref{thm:infl-iter}, it has an
  initial algebra, given by taking the colimit of its inflationary
  iteration.
\end{proof}
To apply this theorem one needs a rich collection of sized
functors. The rest of the section is devoted to exploring closure
properties of sized functors. To do so we use the following operation
on signatures:

\begin{definition}
  \label{def:sig-sum}
  Suppose $\Sigma_c=(A_c,B_c)$ is a family of signatures indexed by
  the elements $c$ of some set $C$. Then the \emph{signature sum}
  $\bigoplus_{c\in C}\Sigma_c$ is the signature $(A,B)$ where
  $A\defeq \sum_{c\in C}A_c = \{(c,a) \mid c\in C \conj a \in A_c\}$
  and $B\in\Sets^A$ maps each $(c,a)$ to the set $B_c(a)$.  As a
  special case when $I=\{0,1\}$, we have the \emph{binary sum}
  $\Sigma_0\oplus\Sigma_1$. There is also an \emph{empty signature}
  $0=(\emptyset,\emptyset)$ which acts as a unit for $\oplus$ up to
  isomorphism (for a suitable notion of signature morphism).
\end{definition}

\begin{remark}
  \label{rem:sig-sum}
  Note that if a size is $(\bigoplus_{c\in C}\Sigma_c)$-filtered, it
  is also $\Sigma_c$-filtered for each $c\in C$. Conversely, given a
  single signature $\Sigma$, if a size is $\Sigma$-filtered, it is
  also $(\bigoplus_{c\in C}\Sigma)$-filtered.
\end{remark}

\begin{proposition}
  \label{prop:sized-closure}
  Suppose that  $\C$,$\D$ and $\E$ are cocomplete categories.
  \begin{enumerate}
  \item\label{item:0} Any cocontinuous functor $\C\fun\D$ is sized.

  \item\label{item:2} Identity functors are sized. If $F:\C\fun\D$ and
    $G:\D\fun\E$ are sized, so is their composition
    $G\comp F: \C\fun \E$.
    
  \item\label{item:3} The terminal functor $\C\fun 1$ and the
    projection functors $\pi_1:\C\times\D\fun \C$ and
    $\pi_2:\C\times\D\fun \C$ are sized; if $F:\C\fun \D$ and
    $G:\C\fun \E$ are sized, then so is $\langle F, G\rangle:\C\fun
    \D\times \E$.

  \item\label{item:4} For any $X\in\C$ the constant functor $1\fun \C$
    with value $X$ is sized.
  \end{enumerate}
\end{proposition}
\begin{proof}
  For part~\ref{item:0}, if $F:\C\fun\D$ is cocontinuous, then it is
  $\Sigma$-sized for any $\Sigma$ and in particular for the empty
  signature.

  The first sentence of part~\ref{item:2} follows from
  part~\ref{item:0}. If $F$ is $\Sigma$-sized and $G$ is
  $\Sigma'$-sized, then $F$ and $G$ both preserve colimits over any
  $\Sigma\oplus\Sigma'$-filtered size, because such a size is also
  $\Sigma$-{} and $\Sigma'$-filtered. The composition $G\comp F$
  preserves such a colimit because $F$ and $G$ do. Therefore $G\comp F$
  is  $\Sigma\oplus\Sigma'$-sized.
  
  For part~\ref{item:3} we use the fact that colimits in a product
  category are computed componentwise. Thus the terminal and
  projection functors are sized by part~\ref{item:0}; and if
  $F:\C\fun \D$ is $\Sigma$-sized and %
  $G:\C\fun \E$ is $\Sigma'$-sized, then $\langle F, G\rangle$ is
  $\Sigma\oplus\Sigma'$-sized.

  For part~\ref{item:4}, note that each size $\kappa$ is directed and
  hence in particular is a connected semi-category; therefore
  $\colim_{i\in\kappa} X$ is canonically isomorphic to $X$. So the
  constant functor with value $X$ is $\Sigma$-sized for any $\Sigma$
  and in particular for the empty signature.
\end{proof}
We can deduce further preservation properties involving infinitary
operations on sized functors by assuming a weak form of choice, which
following \url{https://ncatlab.org/nlab/show/WISC} we call the \wisc{}
axiom. It was introduced in type theory by
Streicher~\cite{StreicherT:reamc} under the name $\mathrm{TTCA}_f$
(``Type Theoretic Collection Axiom'') and independently in
constructive set theory by van~den~Berg and
Moerdijk~\cite{vandenberg2014axiom} under the name ``Axiom of Multiple
Choice''; see also Levy~\cite[Section~5.1]{LevyP:broigp}.
\begin{axiom}[\wisc{}]
  \label{axi:wisc}
  A (possibly large) \emph{cover} of a set $X\in\Sets$ is a surjective
  function $f:Y\surj X$ with $Y\in\Sets_1$.  An indexed
  family\footnote{We will refer to elements of $\Sig$ as
    \emph{families} rather than \emph{signatures} when we are not
    thinking of them as collections of operation symbols of
    set-valued arity.}  $(E_c)_{c\in C} \in \Sig$ is a \emph{wisc} for
  $X\in\Sets$ if for any cover $f:Y\surj X$, there exist $c\in C$ and
  $g : E_c \fun Y$ such that $f\comp g$ is surjective.  The \wisc{}
  axiom\footnote{For simplicity and following \cite{StreicherT:reamc},
    we have given the axiom just for a pair of universes,
    $(\Sets_0,\Sets_1)$; more generally one can ask for the property
    to hold for any pair $(\Sets_m,\Sets_n)$.}  states that for every
  $X\in\Sets$ there exists a family $(E_c)_{c\in C}\in\Sig$ that is a
  wisc for it.
\end{axiom}
``Wisc'' stands for ``weakly initial set of covers'' and the
terminology is justified by the fact that if in $\Sets$ the family
$(E_c)_{c\in C}$ is a wisc for $X$, then the family of covers of $X$
whose domains are of the form $E_c$ for some $c\in C$ is weakly
initial among all the (possibly large) covers of $X$: for every
$Y\in\Sets_1$ and $f:Y\surj X$, there is some cover $e:E_c\surj X$ in
the family that factors as $e=f \comp g$ for some $g : E_c \fun Y$.

Classically, \wisc{} is implied by the Axiom of Choice \ac, since the
latter implies that every surjection has a right inverse and hence the family
whose single member is $X$ is a wisc for $X$. From the results of
van~den~Berg and Moerdijk~\cite{vandenberg2014axiom} (and as noted by
Streicher~\cite{StreicherT:reamc}), if any elementary topos $\E$
satisfies \wisc{}, then so do toposes of (pre)sheaves and
realizability toposes built from $\E$; it is in this sense that the
axiom is constructively acceptable. In particular, starting from the
category of sets in classical set theory with \ac, \wisc{} holds in the
kinds of topos that have been used to model type theory with various
kinds of higher inductive types, whose semantics motivates the work
presented here. (However, it does not hold in all
toposes~\cite{RobertsDM:weacpw}.)

\begin{lemma}[\wisc]
  \label{lem:fam-size}
  Suppose \wisc{} holds and that $\C$ and $\D$ are cocomplete
  categories.  If $(F_x:\C\fun\D)_{x\in X}$ is a family of
  sized functors indexed by a set $X\in\Sets$, then there exists a
  signature $\Sigma\in\Sig$ such that $F_x$ is $\Sigma$-sized for all
  $x\in X$.
\end{lemma}
\begin{proof}
  Consider the large set
  $S \defeq \sum_{x\in X}\{\Sigma'\in\Sig \mid \text{$F_x$ is a
    $\Sigma'$-sized functor}\}$ in $\Sets_1$. By assumption on $F$, the
  first projection $\pi_1:S\fun X$ is a large\footnote{This proof,
    as well as that for \cref{lem:wisc}, illustrates the need for a
    wisc property that quantifies over large covers of small sets.}
  cover of $X$. By \wisc{} there is some surjection $e:X'\surj X$ in
  $\Sets$ and a function $\Sigma':X'\fun \Sig$ so that for all
  $x'\in X'$, the functor $F_{e(x')}$ is $\Sigma'_{x'}$-sized; and
  since $e$ is surjective this implies that each $F_x$ is
  $\Sigma'_{x'}$-sized for some $x'\in X'$. Consider the signature
  $\Sigma \defeq \bigoplus_{x'\in X'}\Sigma'_{x'}$ from
  \cref{def:sig-sum}. By \cref{rem:sig-sum}, each $F_x$ is
  $\Sigma$-sized.
\end{proof}

\begin{theorem}[\wisc{} \textbf{(Colimits of sized functors)}]
  \label{thm:colim}
  Suppose that \wisc{} holds, $\C$ and $\D$ are cocomplete categories,
  $\cat$ is a small category and that $F:\cat\times\C\fun\D$ is a
  functor. If for some signature $\Sigma\in\Sig$ the functor $F(c,\_)$
  is $\Sigma$-sized for each $c\in\cat$, then
  $\colim_{c\in\cat}F(c,\_):\C\fun\D$ is also $\Sigma$-sized. More
  generally, if each $F(c,\_)$ is sized, then so is
  $\colim_{c\in\cat}F(c,\_)$.
\end{theorem}
\begin{proof}
  If $F(c,\_):\C\fun\D$ is $\Sigma$-sized for all $c\in\cat$ and
  $\kappa$ is a $\Sigma$-filtered size, then each $F(c,\_)$ preserves
  colimits of all diagrams $\kappa\fun\C$. Thus given such a diagram
  $D:\kappa\fun\C$, we have a canonical isomorphism
  $F(c,\colim_{i\in\kappa}D_i) \iso \colim_{i\in\kappa}F(c, D_i)$,
  natural in $c$. Taking the colimit over $c\in\cat$ and writing
  $F' \defeq \colim_{c\in\cat}F(c,\_)$, we have
  $
    F'(\colim_{i\in\kappa}D_i) =
    \colim_{c\in\cat}F(c,\colim_{i\in\kappa}D_i) \iso
    \colim_{c\in\cat}\colim_{i\in\kappa} F(c,D_i)
  $.
  Since colimits commute with each other, it follows that the
  canonical morphism
  $F'(\colim_{i\in\kappa}D_i)\fun\colim_{i\in\kappa}F'(D_i)$ is an
  isomorphism. Therefore $F'$ is $\Sigma$-sized. The last sentence of
  the theorem follows by \cref{lem:fam-size}.
\end{proof}

\begin{corollary}[\wisc{}]
  \label{cor:mu}
  Suppose \wisc{} holds and that $\C$ and $\D$ are cocomplete
  categories.  If $F:\C\times\D\fun\D$ is sized, then there is a
  function $X\mapsto \mu Y.F(X,Y)$ assigning to each $X\in\C$ an
  initial algebra $\mu Y.F(X,Y)$ for the functor $F(X,\_):\D\fun\D$.
  The induced functor $\mu Y.F(\_,Y):\C\fun \D$ is sized. \qedhere
\end{corollary}
\begin{proof}
  Suppose $F:\C\times\D\fun\D$ is $\Sigma$-sized. It follows from
  \cref{prop:sized-closure} and \cref{rem:sig-sum} that for each
  $X\in\C$, the functor $F_X \defeq F(X,\_):\D\fun\D$ is
  $\Sigma$-sized. Therefore by \cref{thm:infl-iter,thm:sized-endof},
  the function
  $X\mapsto \mu Y.F(X,Y)\defeq \colim_{i\in \kappa_\Sigma}\mu_i F_X$
  is the required function mapping each $X\in\C$ to an initial algebra
  for $F(X,\_)$. Since each $\mu_i F_X$ is
  $\colim_{j<i}F(X, \mu_j F_X)$, it follows by well-founded induction
  on $i\in\kappa_\Sigma$ that each $\mu_i F_X$ is $\Sigma$-sized,
  using \cref{thm:colim} (taking $\cat$ to be the category generated
  by the thin semi-category $\subsize(i)$). Then by
  \cref{thm:colim} again (taking $\cat$ to be the
  category generated by $\kappa$) we have that
  $\mu Y.F(\_,Y) = \colim_{i\in \kappa_\Sigma}\mu_i F_{\_}$ is
  $\Sigma$-sized.
\end{proof}
Although \cref{prop:sized-closure,thm:colim,cor:mu} show that there is
quite a rich collection of sized functors, what is lacking so far is
any closure under taking limits, assuming the target category has
them; in other words the dual of \cref{thm:colim}. We consider this
for the case $\D=\Sets$, leaving consideration of more general
complete and cocomplete categories for future work. First note that if
$F,G:\C\fun\Sets$ are sized functors, the equalizer of any parallel
pair $F\rightrightarrows G$ of natural transformations is also a sized
functor (it is $(\Sigma\oplus\Sigma')$-sized if $F$ is $\Sigma$-sized
and $G$ is $\Sigma'$-sized). This is because each size $\kappa$ is
directed and so taking $\kappa$-colimits in $\Sets$ commutes with
finite limits and hence in particular with equalizers. So to get
closure of sized functors under all small limits it suffices to
consider small products. For this we need to use a ``double cover''
signature of a set (the wiscs $W$ and $W'$ in the proof of
\cref{thm:product} below), inspired by the use that
Swan~\cite{swan2018wtypes} makes of the indexed form of the WISC
Axiom; see also \cite{PittsAM:quoitq}.  So we will need wiscs for
indexed families of sets; but their existence follows from \wisc:
\begin{lemma}[\wisc{}]
  \label{lem:wisc}
  Assuming \wisc{} holds, then for every family of sets
  $(X_i)_{i\in I}\in\Sig$ there exists a family
  $(E_c)_{c\in C}\in\Sig$ that is a wisc for each set $X_i$.
\end{lemma}
\begin{proof}
  Consider $S \defeq \sum_{i\in I}\{W\in\Sig \mid \text{$W$ is a wisc
    for $X_i$}\} \in \Sets_1$. By \wisc{}, the first projection
  $\pi_1:S\fun I$ is a large cover of $I$. Since there is a wisc for
  $I$, it follows that there is some surjection $e:J\surj I$ in
  $\Sets$ and a function $W:J\fun \Sig$ so that for all $j\in J$,
  $W_j$ is a wisc for $X_{e(j)}$. Consider the signature sum
  $W\defeq\bigoplus_{j\in J}W_j\in\Sig$ as in \cref{def:sig-sum}. Thus
  writing $(C,E)$ for $W$ and $(C_j,E_j)$ for each $W_j$, we have
  $C\defeq \sum_{j\in J} C_j \in \Sets$ and $E\in \Sets^C$ is the
  function mapping each $(j,c)\in \sum_{j\in J} C_j$ to $E_j(c)$. Then
  we claim that $W\in\Sig$ is a wisc for each set $X_i$. For,
  given any cover $f:Y\surj X_i$, since $e:J\surj I$ is a surjection,
  there exists $j\in J$ with $e(j)= i$; then since $W_j=(C_j,E_j)$ is a
  wisc for $X_{e(j)} = X_i$, there exists $c\in C_j$ and
  $g:E_j(c)\fun Y$ such that $f\comp g$ is surjective. So there exists
  $(j,c)\in C$ and $g:E(j,c) = E_j(c) \fun Y$ such that $f\comp g$ is
  surjective. Therefore $W=(C,E)$ does indeed have the wisc property for
  $X_i$.
\end{proof}

\begin{theorem}[\wisc{} \textbf{(Products of set-valued sized functors are
  sized)}]
  \label{thm:product}
  Suppose that $\C$ is a cocomplete category.  Assuming \wisc{}
  holds, if $(F_x:\C\fun\Sets)_{x\in X}$ is a family of sized functors
  indexed by some set $X\in\Sets$, then the functor
  $\prod_{x\in X}F_x : \C\fun\Sets$ given by taking products in
  $\Sets$ is also sized.
\end{theorem}
\begin{proof}
  By \cref{lem:fam-size}, there exists a signature $\Sigma$ so that
  each functor $F_x$ is $\Sigma$-sized. However, we need a bigger
  signature than $\Sigma$ in order to prove that $\prod_{x\in X}F_x$
  is sized. Using \wisc{}, let $W = (E_c)_{c\in C}$ be a wisc for
  $X$. Then using \cref{lem:wisc}, let $W' = (E'_{c'})_{c'\in C'}$ be
  a wisc for the sets in the family
  $(\ker p)_{c\in C, p: E_c\surj X}$, where
  \begin{equation}
    \label{eq:product-0}
    \ker p \defeq \{(d_1,d_2)\in E_c\times E_c \mid p(d_1) = p(d_2)\}
  \end{equation}
  We claim that the functor $F'\defeq\prod_{x\in X}F_x$ is
  $\Sigma'$-sized when $\Sigma' = \Sigma \oplus W \oplus W'$ (using
  the signature sum from \cref{def:sig-sum}).

  If $D:\kappa\fun \C$ is a diagram on a $\Sigma'$-filtered size
  $\kappa$, then by \cref{rem:sig-sum}, each $F_x$ is $\Sigma'$-sized
  and so we have a canonical isomorphism
  $\colim_{i\in\kappa} F_x(D_i) \iso F_x(\colim_{i\in\kappa}
  D_i)$. Taking the product over $x\in X$, we get
  $\prod_{x\in X} \colim_{i\in\kappa} F_x(D_i) \iso \prod_{x\in X}
  F_x(\colim_{i\in\kappa} D_i) = (\prod_{x\in X}
  F_x)(\colim_{i\in\kappa} D_i)$.  So it just remains to show that the
  canonical function
  \begin{equation}
    \label{eq:product-1}
    \textstyle
     \can_{F,D} : \colim_{i\in\kappa}\left(\left(\prod_{x\in X}
        F_x\right) D_i\right) = \colim_{i\in\kappa}\prod_{x\in X}
    F_x(D_i) \fun \prod_{x\in X}\colim_{i\in\kappa}F_x(D_i)
  \end{equation}
  is an isomorphism, that is, both an injection and a surjection. The
  summand $W$ in $\Sigma'$ ensures that $\kappa$ has upper bounds for
  $E_c$-indexed families for any $c\in C$; and the $W'$ summand
  ensures the same for $E'_{c'}$-indexed families, for any
  $c' \in C'$. The first kind of upper bound, together with the wisc
  property of $W$, comes into play in proving that $\can_{F,D}$ is
  injective; and both kinds of upper bounds and the wisc property of
  $W$ and $W'$ come into play in proving that $\can_{F,D}$ is
  surjective.

  To prove that $\can_{F,D}$ is injective and surjective we use the
  fact that the colimit in $\Sets$ of a directed diagram
  $D:\kappa\fun\Sets$ can be described explicitly as the quotient
  $(\sum_{i\in\kappa}D_i)/{\approx}$ where the equivalence relation
  $\approx$ identifies $(i,d),(i',d') \in \sum_{i\in\kappa}D_i$ if
  there is some $j\in\kappa$ with $i<j$, $i'<j$ and
  $D_{i,j}(d) = D_{i',j}(d')$. We will write $[i,d]_{\approx}$ for the
  ${\approx}$-equivalence class of $(i,d)\in
  \sum_{i\in\kappa}D_i$. Then the function in \cref{eq:product-1}
  satisfies for all $i\in \kappa$ and $f\in \prod_{x\in X} F_x(D_i)$
  \[
  \can_{F,D}[i, f]_{\approx} = \lambda x\in X.\; [i,f(x)]_{\approx}
  \]
  To see that $\can_{F,D}$ is injective, suppose we also have
  $i'\in \kappa$ and $f'\in \prod_{x\in X} F_x(D_i)$ satisfying
  $\forall x \in X.\; [i,f(x)]_{\approx} = [i',f'(x)]_{\approx}$; we
  wish to prove that $(i,f)\approx(i',f')$. By definition of $\approx$
  we have
  $\forall x\in X.\exists j\in\kappa.\; i < j \;\conj\; i' < j
  \;\conj\; F_x(D_{i,j})(f\,x) = F_x(D_{i',j})(f'x)$. Since
  $W = (E_c)_{c\in C}$ is a wisc for $X$, there exist $c\in C$, a
  surjection $p:E_c\surj X$ and a function $q:E_c\fun \kappa$ so that
  \begin{equation}
    \label{eq:product-2}
    \forall z\in E_c.\; i< q(z) \;\conj\; i'< q(z) \;\conj\;
    F_x(D_{i,q(z)})(f(p\,z)) = F_x(D_{i',q(z)})(f'(p\,z))
  \end{equation}
  Since $W$ is a summand in $\Sigma'$ and $\kappa$ is a
  $\Sigma'$-filtered size, there is an $<$-upper bound $j\in\kappa$
  for $q:E_c\fun \kappa$; and since $\kappa$ is directed, we can
  assume $i<j$ and $i'<j$. So from \eqref{eq:product-2} and
  surjectiviy of $p$ we deduce that
  $\forall x\in X.\; F_x(D_{i,j})(f\,x) = F_x(D_{i',j})(f'x)$, which implies
  $(i,f)\approx (i',f')$. Therefore the function
  $\can_{F,D}$ in \eqref{eq:product-1} is indeed injective.

  To see that $\can_{F,D}$ is also surjective, suppose we have
  $g\in\prod_{x\in X}\colim_{i\in\kappa}F_x(D_i)$. Since
  \[\textstyle
    \forall x\in X.\exists(i,d)\in \sum_{i\in\kappa}F_x(D_i).\; g(x) =
    [i,d]_{\approx}
  \]
  and $W$ is a wisc for $X$, there exists some $c\in C$,
  $p:E_c\surj X$ and
  $\langle q_1,q_2\rangle \in \prod_{z\in E_c}
  \sum_{i\in\kappa}F_{p(z)}(D_i)$ so that
  $\forall z\in E_c.\; g(p\,z) =[q_1(z),q_2(z)]_{\approx}$. Then since
  $W$ is a summand in $\Sigma'$ and $\kappa$ is a $\Sigma'$-filtered
  size, there is an $<$-upper bound $j\in\kappa$ for
  $q_1:E_c\fun \kappa$. So we have
  \begin{equation}
    \label{eq:product-3}
    \forall z\in E_c.\; g(p\,z) = [j, q'(z)]_{\approx}
 \end{equation}
 where $q'\in\prod_{z\in E_c}F_{p(z)}(D_j)$ is
 $q'(z)\defeq F_{p(z)}(D_{q_1(z),j})(q_2(z))$. It follows that the
 relation $\Phi\subseteq \sum_{x\in X} F_x(D_j)$ given by
 $\Phi(x,d) \defeq \exists z\in E_c.\; x = p(z) \;\conj\; d = q'(z)$
 is total (because $p$ is surjective); and were it also single-valued,
 it would determine a function $f\in\prod_{x\in X} F_x(D_j)$ which by
 virtue of \eqref{eq:product-3} would satisfy $\can_{F,D}[j,f] =
 g$. However, we need to increase $j$ to get this single-valued
 property. Recall that $W'$ is a wisc for the kernel
 \eqref{eq:product-0} of $p:E_c\surj X$.  If $(z_1,z_2)\in \ker p$,
 then by \eqref{eq:product-3}
 $[j,q'(z_1)]_{\approx} = g(p\,z_1) = g(p\,z_2) =
 [j,q'(z_2)]_{\approx}$. Therefore we have
 \[
   \forall (z_1,z_2)\in\ker p.\exists k\in \kappa.\; j < k \;\conj\;
   F_{p(z_1)}(D_{q_1(z_1),k})(q_2(z_1)) = F_{p(z_2)}(D_{q_1(z_2),k})(q_2(z_2))
 \]
 So since $W'=(E_{c'})_{c'\in C'}$ is a wisc for $\ker p$ and $\kappa$
 has $E_{c'}$-indexed upper bounds for any $c'\in C'$ and is directed,
 it follows that there exists $c'\in C'$, 
 $\langle p_1,p_2\rangle:E_{c'}\surj \ker p$ and $k\in\kappa$ with
 $j<k$ and
 \begin{equation}
   \label{eq:product-4}
   \forall z''\in E_{c'}. F_{p(p_1\,z'')}(D_{q_1(p_1 z''),k})(q_2(p_1\,z'')) =
   F_{p(p_2\,z'')}(D_{q_1(p_2 z''),k})(q_2(p_2\,z''))
 \end{equation}
 Now if we let $q''\in\prod_{z\in E_c}F_{p(z)}(D_k)$ be
 $q''(z)\defeq F_{p(z)}(D_{q_1(z),k})(q_2(z))$, then from
 \eqref{eq:product-3} we have
 \begin{equation}
   \label{eq:product-5}
    \forall z\in E_c.\; g(p\,z) = [k, q''(z)]_{\approx}
  \end{equation}
  Let the relation $\Phi'\subseteq \sum_{x\in X} F_x(D_k)$ be
  $\Phi'(x,d) \defeq \exists z\in E_c.\; x = p(z) \;\conj\; d =
  q''(z)$.  It is total because $p$ is surjective; but it is also
  single-valued because if $\Phi'(x,d) \conj \Phi'(x,d')$, then
  $d = q''(z_1) \conj d' = q''(z_2)$ for some $(z_1,z_2)\in \ker p$,
  so that there exists $z''\in E_{c'}$ with
  $p_1(z'')=z_1 \conj p_2(z'') = z_2$ and hence
  $d=q''(z_1) = q''(z_2) = d'$ by \eqref{eq:product-4}. Therefore
  $\Phi'$ is the graph of a function $f\in\prod_{x\in X} F_x(D_k)$;
  and by virtue of \eqref{eq:product-5} we have
  $\forall x\in X.\; g(x) = [k,f(x)]_{\approx}$, so that
  $g = \can_{F,D}[k,f]$. Thus $\can_{F,D} $ is indeed surjective.
\end{proof}

\begin{example}
  \label{exa:sym-cont}
  The \emph{symmetric containers} of Gylterud~\cite{GylterudHR:symc}
  generalize ordinary signatures by replacing the set of operation
  symbols by a groupoid $\mathbf{A}$ and the arity function by a
  functor $B:\mathbf{A}\fun\Sets$. The associated endofunctor
  $S_{\mathbf{A},B}:\Sets\fun\Sets$ maps each set $X\in\Sets$ to the
  colimit
  \begin{equation}
    \label{eq:symm-cont}
    S_{\mathbf{A},B}(X) \defeq \colim_{a\in\mathbf{A}} X^{B(a)}
  \end{equation}
  Applying \cref{thm:sized-endof}, \cref{prop:sized-closure} and
  \cref{thm:product} we have that any topos with universes satisfying
  \wisc{} has initial algebras for symmetric containers.

  In fact these initial algebras are special cases of
  QW-types~\cite{PittsAM:quoitq}: they can be seen as sets of terms
  quotiented by the symmetries given by the groupoid structure on the
  arguments of an operation symbol. So their existence in toposes with
  \wisc{} follows from the results of that paper. However, the
  construction here in terms of a colimit of an inflationary iteration
  gives a simpler description than for the general case of a QW-type.
\end{example}

\section{Related and future work}
\label{sec:relw}

The results in this paper make use of the constructive techniques
introduced by the authors and Fiore in our prior paper
\cite{PittsAM:quoitq}: the use of sizes given by ``plump''
well-founded orders on W-types and the use of a WISC axiom to see that
certain functors preserve colimits of that shape. That paper
constructs a large class of quotient-inductive types, called
\emph{QWI-types}, which by definition are initial among algebras for
indexed containers~\cite{abbott2005containers} satisfying a given
system of equations. Although the construction proceeds by forming a
size-indexed family of objects in the case $\C$ is $\Sets^I$ (with
$I\in\Sets$) and taking its colimit, it does not appear to be a direct
corollary of \cref{thm:infl-iter}. Conversely, the results here do not
follow from the ones in \cite{PittsAM:quoitq}, since for one thing
here we consider general cocomplete categories $\C$, rather than just
products of $\Sets$. In this respect we are closer to the approach of
Fiore and Hur~\cite{FioreMP:confae} and it would be interesting to see
whether our techniques can be extended to give constructive proofs of
existence of free algebras for the very general notion of equational
system on a category that is introduced in that paper. This may
involve investigating the extent to which our approach allows a
constructive treatment of some of the classical theory of locally
presentable and accessible categories~\cite{AdamekJ:locpac}, which is
future work.

The inflationary iteration indexed by a notion of size that we have
introduced in the paper generalises from complete posets to
cocomplete categories aspects of Abel and Pientka's
work~\cite{AbelA:typbti,AbelA:wellfrc}. These papers develop a theory
of \emph{sized types} and its semantics. Abel has added a version of
this to the type theory provided by the Agda proof
assistant~\cite{Agda261}. Unfortunately recent versions of Agda
contain features that allow one to use sized types to prove a logical
contradiction. The problem is that, in contrast to the notion of size
used here, the one by Abel et al.~\cite{AbelA:typbti,AbelA:wellfrc}
features a generic size $\infty$ at which sized-indexed sequences
become stationary. Currently in Agda (version~2.6.2) one both has
$\infty < \infty$ and can prove that $<$ is well-founded, leading to a
contradiction. For us, the intuitive and important aspect of ``size''
is that there is well-founded ordering, thus permitting definitions by
well-founded recursion on a set of sizes.  Then having a single size
$\infty$ at which all sequences become stationary is semantically
problematic. So we avoid having an explicit stationary size $\infty$,
at the expense of having to take a colimit to obtain an initial
algebra, instead of just instantiating an inflationary iteration at
$\infty$.

We hope Agda's sized types will get fixed, since they are useful in
practice; they are most often used (together with copatterns) to
demonstrate that recursively defined functions on a coinductively
defined record type are well-defined (that is, are
``productive'')~\cite{AbelA:wellfrc}. Here, while avoiding sized
types, we can still dualise \cref{thm:infl-iter}. Applying it
to the opposite category $\C^{\op}$, we have that if $\C$ is complete
and $F:\C\fun\C$ preserves limits of diagrams $\kappa\fun\C$ for some
size $\kappa$, then \emph{$F$ has a final coalgebra $\nu F$ given by
  the limit of a deflationary iteration
  $(\nu_iF = \lim_{j<i}F(\nu_j F))_{i\in\kappa}$}. We have yet to
investigate whether this is useful, that is, how rich the class of
such endofunctors is in a constructive setting.

Ad\'amek, Milius and Moss~\cite{AdamekJ:iniatw} take a different
approach to constructive initial algebra theorems than the one here,
avoiding iteration of the endofunctor. They consider categories $\C$
with colimits of diagrams of monomorphisms (from some well-behaved
class) and endofunctors $F:\C\fun\C$ that preserve those
monomorphisms. Using the intuitionistically valid fixed point theorem
of Pataraia (see~\cite[Theorem~3.2]{BauerA:bouwpt}), they prove that
such an $F$ has an initial algebra iff it has a prefixed point (an
algebra whose structure morphism is a monomorphism). Preserving
monomorphisms seems less of a condition on a functor than the one we
need for \cref{thm:infl-iter}, that is, preserving colimits of some
size $\kappa$ (although the two conditions are independent). However,
as we saw in \cref{thm:colim}, our class of sized endofunctors is
closed under taking coequalizers, so that we get initial algebras for
constructs involving quotients, such as \cref{exa:sym-cont}, whereas
endofunctors preserving monomorphisms are not in general closed under
taking coequalizers. Another difference to~\cite{AdamekJ:iniatw} is
that it uses impredicative principles (the proof of Pataraia's fixed
point theorem uses impredicative quantification), whereas our Agda
development~\cite{agdacode} shows that our initial algebra theorem
(\cref{thm:infl-iter}) is valid in a predicative constructive logic.


%% file: root.bbl
\begin{thebibliography}{10}
\providecommand{\bibitemdeclare}[2]{}
\providecommand{\surnamestart}{}
\providecommand{\surnameend}{}
\providecommand{\urlprefix}{Available at }
\providecommand{\url}[1]{\texttt{#1}}
\providecommand{\href}[2]{\texttt{#2}}
\providecommand{\urlalt}[2]{\href{#1}{#2}}
\providecommand{\doi}[1]{doi:\urlalt{http://dx.doi.org/#1}{#1}}
\providecommand{\bibinfo}[2]{#2}

\bibitemdeclare{article}{abbott2005containers}
\bibitem{abbott2005containers}
\bibinfo{author}{M.~\surnamestart Abbott\surnameend},
  \bibinfo{author}{T.~\surnamestart Altenkirch\surnameend} \&
  \bibinfo{author}{N.~\surnamestart Ghani\surnameend} (\bibinfo{year}{2005}):
  \emph{\bibinfo{title}{Containers: Constructing Strictly Positive Types}}.
\newblock {\sl \bibinfo{journal}{Theoretical Computer Science}}
  \bibinfo{volume}{342}(\bibinfo{number}{1}), pp. \bibinfo{pages}{3--27},
  \doi{10.1016/j.tcs.2005.06.002}.

\bibitemdeclare{article}{AbelA:typbti}
\bibitem{AbelA:typbti}
\bibinfo{author}{A.~\surnamestart Abel\surnameend} (\bibinfo{year}{2012}):
  \emph{\bibinfo{title}{Type-Based Termination, Inflationary Fixed-Points, and
  Mixed Inductive-Coinductive Types}}.
\newblock {\sl \bibinfo{journal}{Electronic Proceedings in Theoretical Computer
  Science}} \bibinfo{volume}{77}, pp. \bibinfo{pages}{1--11},
  \doi{10.4204/EPTCS.77.1}.

\bibitemdeclare{article}{AbelA:wellfrc}
\bibitem{AbelA:wellfrc}
\bibinfo{author}{A.~\surnamestart Abel\surnameend} \&
  \bibinfo{author}{B.~\surnamestart Pientka\surnameend} (\bibinfo{year}{2016}):
  \emph{\bibinfo{title}{Well-Founded Recursion with Copatterns and Sized
  Types}}.
\newblock {\sl \bibinfo{journal}{Journal of Functional Programming}}
  \bibinfo{volume}{26}, p.~\bibinfo{pages}{61},
  \doi{10.1017/S0956796816000022}.

\bibitemdeclare{article}{AdamekJ:freaar}
\bibitem{AdamekJ:freaar}
\bibinfo{author}{J.~\surnamestart Ad{\'a}mek\surnameend}
  (\bibinfo{year}{1974}): \emph{\bibinfo{title}{Free Algebras and Automata
  Realizations in the Language of Categories}}.
\newblock {\sl \bibinfo{journal}{Commentationes Mathematicae Universitatis
  Carolinae}} \bibinfo{volume}{15}(\bibinfo{number}{4}), pp.
  \bibinfo{pages}{589--602}.

\bibitemdeclare{article}{AdamekJ:iniatw}
\bibitem{AdamekJ:iniatw}
\bibinfo{author}{J.~\surnamestart Ad{\'a}mek\surnameend},
  \bibinfo{author}{S.~\surnamestart Milius\surnameend} \&
  \bibinfo{author}{L.~S. \surnamestart Moss\surnameend} (\bibinfo{year}{2021}):
  \emph{\bibinfo{title}{An Initial Algebra Theorem Without Iteration}}.
\newblock {\sl \bibinfo{journal}{ArXiv e-prints}}
  \bibinfo{volume}{arXiv:2104.09837 [cs.LO]}.
\newblock \urlprefix\url{https://arxiv.org/abs/2104.09837}.

\bibitemdeclare{unpublished}{ammbook}
\bibitem{ammbook}
\bibinfo{author}{J.~\surnamestart Ad{\'a}mek\surnameend},
  \bibinfo{author}{S.~\surnamestart Milius\surnameend} \&
  \bibinfo{author}{L.~S. \surnamestart Moss\surnameend} (\bibinfo{year}{2021}):
  \emph{\bibinfo{title}{Initial Algebras, Terminal Coalgebras, and the Theory
  of Fixed Points of Functors}}.
\newblock \urlprefix\url{http://www.stefan-milius.eu}.
\newblock \bibinfo{note}{Draft book}.

\bibitemdeclare{book}{AdamekJ:locpac}
\bibitem{AdamekJ:locpac}
\bibinfo{author}{J.~\surnamestart Ad{\'a}mek\surnameend} \&
  \bibinfo{author}{J.~\surnamestart Rosick{\'y}\surnameend}
  (\bibinfo{year}{1994}): \emph{\bibinfo{title}{Locally Presentable and
  Accessible Categories}}.
\newblock \bibinfo{series}{London Mathematical Society Lecture Note Series},
  \bibinfo{publisher}{Cambridge University Press},
  \doi{10.1017/CBO9780511600579}.

\bibitemdeclare{misc}{Agda261}
\bibitem{Agda261}
\bibinfo{author}{\surnamestart {Agda v2.6.1}\surnameend}
  (\bibinfo{year}{2021}):
  \urlprefix\url{https://agda.readthedocs.io/en/v2.6.1.3/index.html}.

\bibitemdeclare{inproceedings}{AltenkirchT:quoiit}
\bibitem{AltenkirchT:quoiit}
\bibinfo{author}{T.~\surnamestart Altenkirch\surnameend},
  \bibinfo{author}{P.~\surnamestart Capriotti\surnameend},
  \bibinfo{author}{G.~\surnamestart Dijkstra\surnameend},
  \bibinfo{author}{N.~\surnamestart Kraus\surnameend} \& \bibinfo{author}{F.~N.
  \surnamestart Forsberg\surnameend} (\bibinfo{year}{2018}):
  \emph{\bibinfo{title}{Quotient Inductive-Inductive Types}}.
\newblock In \bibinfo{editor}{C.~\surnamestart Baier\surnameend} \&
  \bibinfo{editor}{U.~Dal \surnamestart Lago\surnameend}, editors: {\sl
  \bibinfo{booktitle}{Foundations of Software Science and Computation
  Structures, {FoSSaCS}~2018}}, {\sl \bibinfo{series}{Lecture Notes in Computer
  Science}} \bibinfo{volume}{10803}, \bibinfo{publisher}{Springer International
  Publishing}, pp. \bibinfo{pages}{293--310},
  \doi{10.1007/978-3-319-89366-2{\_}16}.

\bibitemdeclare{article}{BauerA:bouwpt}
\bibitem{BauerA:bouwpt}
\bibinfo{author}{A~\surnamestart Bauer\surnameend} \&
  \bibinfo{author}{P.~\surnamestart Lumsdaine\surnameend}
  (\bibinfo{year}{2013}): \emph{\bibinfo{title}{On the Bourbaki–Witt
  Principle in Toposes}}.
\newblock {\sl \bibinfo{journal}{Mathematical Proceedings of the Cambridge
  Philosophical Society}} \bibinfo{volume}{155}(\bibinfo{number}{1}), pp.
  \bibinfo{pages}{87--99}, \doi{10.1017/S0305004113000108}.

\bibitemdeclare{article}{DybjerP:genfsi}
\bibitem{DybjerP:genfsi}
\bibinfo{author}{P.~\surnamestart Dybjer\surnameend} (\bibinfo{year}{2000}):
  \emph{\bibinfo{title}{A General Formulation of Simultaneous
  Inductive-Recursive Definitions in Type Theory}}.
\newblock {\sl \bibinfo{journal}{Journal of Symbolic Logic}}
  \bibinfo{volume}{65}(\bibinfo{number}{2}), pp. \bibinfo{pages}{525--549},
  \doi{10.1305/ndjfl/1093635159}.

\bibitemdeclare{article}{FioreMP:confae}
\bibitem{FioreMP:confae}
\bibinfo{author}{M.~P. \surnamestart Fiore\surnameend} \&
  \bibinfo{author}{C.-K. \surnamestart Hur\surnameend} (\bibinfo{year}{2008}):
  \emph{\bibinfo{title}{On the Construction of Free Algebras for Equational
  Systems}}.
\newblock {\sl \bibinfo{journal}{Theoretical Computer Science}}
  \bibinfo{volume}{410}, pp. \bibinfo{pages}{1704--1729},
  \doi{10.1016/j.tcs.2008.12.052}.

\bibitemdeclare{inproceedings}{PittsAM:coniqi}
\bibitem{PittsAM:coniqi}
\bibinfo{author}{M.~P. \surnamestart Fiore\surnameend}, \bibinfo{author}{A.~M.
  \surnamestart Pitts\surnameend} \& \bibinfo{author}{S.~C. \surnamestart
  Steenkamp\surnameend} (\bibinfo{year}{2020}):
  \emph{\bibinfo{title}{Constructing Infinitary Quotient-Inductive Types}}.
\newblock In \bibinfo{editor}{J.~\surnamestart Goubault-Larrecq\surnameend} \&
  \bibinfo{editor}{B.~\surnamestart K{\"o}nig\surnameend}, editors: {\sl
  \bibinfo{booktitle}{23rd International Conference on Foundations of Software
  Science and Computation Structures (FoSSaCS 2020)}}, {\sl
  \bibinfo{series}{Lecture Notes in Computer Science}} \bibinfo{volume}{12077},
  \bibinfo{publisher}{Springer}, pp. \bibinfo{pages}{257--276},
  \doi{10.1007/978-3-030-45231-5_14}.

\bibitemdeclare{article}{PittsAM:quoitq}
\bibitem{PittsAM:quoitq}
\bibinfo{author}{M.~P. \surnamestart Fiore\surnameend}, \bibinfo{author}{A.~M.
  \surnamestart Pitts\surnameend} \& \bibinfo{author}{S.~C. \surnamestart
  Steenkamp\surnameend} (\bibinfo{year}{2021}):
  \emph{\bibinfo{title}{Quotients, Inductive Types and Quotient Inductive
  Types}}.
\newblock {\sl \bibinfo{journal}{ArXiv e-prints}}
  \bibinfo{volume}{arXiv:2101.02994 [cs.LO]}.
\newblock \urlprefix\url{https://arxiv.org/abs/2101.02994}.

\bibitemdeclare{phdthesis}{ForsbergFN:indid}
\bibitem{ForsbergFN:indid}
\bibinfo{author}{F.~N. \surnamestart Forsberg\surnameend}
  (\bibinfo{year}{2013}): \emph{\bibinfo{title}{Inductive-Inductive
  Definitions}}.
\newblock Ph.D. thesis, \bibinfo{school}{Swansea University}.

\bibitemdeclare{inproceedings}{GambinoN:welftd}
\bibitem{GambinoN:welftd}
\bibinfo{author}{N.~\surnamestart Gambino\surnameend} \&
  \bibinfo{author}{M.~\surnamestart Hyland\surnameend} (\bibinfo{year}{2004}):
  \emph{\bibinfo{title}{Wellfounded Trees and Dependent Polynomial Functors}}.
\newblock In \bibinfo{editor}{S.~\surnamestart Berardi\surnameend},
  \bibinfo{editor}{M.~\surnamestart Coppo\surnameend} \&
  \bibinfo{editor}{F.~\surnamestart Damiani\surnameend}, editors: {\sl
  \bibinfo{booktitle}{Types for Proofs and Programs}}, \bibinfo{series}{Lecture
  Notes in Computer Science}, \bibinfo{publisher}{Springer Berlin Heidelberg},
  \bibinfo{address}{Berlin, Heidelberg}, pp. \bibinfo{pages}{210--225},
  \doi{10.1007/978-3-540-24849-1{\_}14}.

\bibitemdeclare{mastersthesis}{GylterudHR:symc}
\bibitem{GylterudHR:symc}
\bibinfo{author}{H.~R. \surnamestart Gylterud\surnameend}
  (\bibinfo{year}{2011}): \emph{\bibinfo{title}{Symmetric Containers}}.
\newblock \bibinfo{type}{Master of {S}cience}, \bibinfo{school}{Department of
  Mathematics, University of Oslo}.
\newblock
  \urlprefix\url{https://www.duo.uio.no/bitstream/handle/10852/10740/thesisgylterud.pdf}.

\bibitemdeclare{book}{JohnstonePT:skeett}
\bibitem{JohnstonePT:skeett}
\bibinfo{author}{P.~T. \surnamestart Johnstone\surnameend}
  (\bibinfo{year}{2002}): \emph{\bibinfo{title}{Sketches of an Elephant, A
  Topos Theory Compendium, Volumes 1 and 2}}.
\newblock {\sl \bibinfo{series}{Oxford Logic Guides}} \bibinfo{volume}{43--44},
  \bibinfo{publisher}{Oxford University Press}.

\bibitemdeclare{inproceedings}{KaposiA:lariqi}
\bibitem{KaposiA:lariqi}
\bibinfo{author}{A.~\surnamestart Kov\'{a}cs\surnameend} \&
  \bibinfo{author}{A.~\surnamestart Kaposi\surnameend} (\bibinfo{year}{2020}):
  \emph{\bibinfo{title}{Large and Infinitary Quotient Inductive-Inductive
  Types}}.
\newblock In: {\sl \bibinfo{booktitle}{Proceedings of the 35th Annual ACM/IEEE
  Symposium on Logic in Computer Science}}, \bibinfo{series}{LICS '20},
  \bibinfo{publisher}{Association for Computing Machinery},
  \bibinfo{address}{New York, NY, USA}, p. \bibinfo{pages}{648–661},
  \doi{10.1145/3373718.3394770}.

\bibitemdeclare{article}{LevyP:broigp}
\bibitem{LevyP:broigp}
\bibinfo{author}{P.~B. \surnamestart Levy\surnameend} (\bibinfo{year}{2021}):
  \emph{\bibinfo{title}{Broad Infinity and Generation Principles}}.
\newblock {\sl \bibinfo{journal}{ArXiv e-prints}}
  \bibinfo{volume}{arXiv:2101.01698 [math.LO]}.
\newblock \urlprefix\url{https://arxiv.org/abs/2101.01698}.

\bibitemdeclare{inproceedings}{PittsAM:intumh}
\bibitem{PittsAM:intumh}
\bibinfo{author}{D.~R. \surnamestart Licata\surnameend},
  \bibinfo{author}{I.~\surnamestart Orton\surnameend}, \bibinfo{author}{A.~M.
  \surnamestart Pitts\surnameend} \& \bibinfo{author}{B.~\surnamestart
  Spitters\surnameend} (\bibinfo{year}{2018}): \emph{\bibinfo{title}{Internal
  Universes in Models of Homotopy Type Theory}}.
\newblock In \bibinfo{editor}{H.~\surnamestart Kirchner\surnameend}, editor:
  {\sl \bibinfo{booktitle}{3rd International Conference on Formal Structures
  for Computation and Deduction (FSCD 2018)}}, {\sl \bibinfo{series}{Leibniz
  International Proceedings in Informatics (LIPIcs)}} \bibinfo{volume}{108},
  \bibinfo{publisher}{Schloss Dagstuhl--Leibniz-Zentrum f\"ur Informatik},
  \bibinfo{address}{Dagstuhl, Germany}, pp. \bibinfo{pages}{22:1--22:17},
  \doi{10.4230/LIPIcs.FSCD.2018.22}.

\bibitemdeclare{book}{Martin-LoefP:inttt}
\bibitem{Martin-LoefP:inttt}
\bibinfo{author}{P.~\surnamestart {Martin-L\"of}\surnameend}
  (\bibinfo{year}{1984}): \emph{\bibinfo{title}{Intuitionistic Type Theory}}.
\newblock \bibinfo{publisher}{Bibliopolis, Napoli}.

\bibitemdeclare{inproceedings}{MeijerE:funpbl}
\bibitem{MeijerE:funpbl}
\bibinfo{author}{E.~\surnamestart Meijer\surnameend},
  \bibinfo{author}{M.~\surnamestart Fokkinga\surnameend} \&
  \bibinfo{author}{R.~\surnamestart Paterson\surnameend}
  (\bibinfo{year}{1991}): \emph{\bibinfo{title}{Functional Programming with
  Bananas, Lenses, Envelopes and Barbed Wire}}.
\newblock In \bibinfo{editor}{J.~\surnamestart Hughes\surnameend}, editor: {\sl
  \bibinfo{booktitle}{Functional Programming Languages and Computer
  Architecture, {FPCA}~1991}}, {\sl \bibinfo{series}{Lecture Notes in Computer
  Science}} \bibinfo{volume}{523}, \bibinfo{publisher}{Springer, Berlin,
  Heidelberg}, pp. \bibinfo{pages}{124--144}, \doi{10.1007/3540543961{\_}7}.

\bibitemdeclare{article}{MoerdijkI:weltc}
\bibitem{MoerdijkI:weltc}
\bibinfo{author}{I.~\surnamestart Moerdijk\surnameend} \&
  \bibinfo{author}{E.~\surnamestart Palmgren\surnameend}
  (\bibinfo{year}{2000}): \emph{\bibinfo{title}{Wellfounded Trees in
  Categories}}.
\newblock {\sl \bibinfo{journal}{Annals of Pure and Applied Logic}}
  \bibinfo{volume}{104}(\bibinfo{number}{1}), pp. \bibinfo{pages}{189--218},
  \doi{10.1016/S0168-0072(00)00012-9}.

\bibitemdeclare{article}{MoerdijkI:typttc}
\bibitem{MoerdijkI:typttc}
\bibinfo{author}{I.~\surnamestart Moerdijk\surnameend} \&
  \bibinfo{author}{E.~\surnamestart Palmgren\surnameend}
  (\bibinfo{year}{2002}): \emph{\bibinfo{title}{Type theories, Toposes and
  Constructive Set Theory: Predicative Aspects of {AST}}}.
\newblock {\sl \bibinfo{journal}{Annals of Pure and Applied Logic}}
  \bibinfo{volume}{114}(\bibinfo{number}{1}), pp. \bibinfo{pages}{155--201},
  \doi{10.1016/S0168-0072(01)00079-3}.

\bibitemdeclare{book}{NordstromB:progmlt}
\bibitem{NordstromB:progmlt}
\bibinfo{author}{B.~\surnamestart Nordstr{\"o}m\surnameend},
  \bibinfo{author}{K.~\surnamestart Petersson\surnameend} \&
  \bibinfo{author}{J.~M. \surnamestart Smith\surnameend}
  (\bibinfo{year}{1990}): \emph{\bibinfo{title}{Programming in
  {Martin-L{\"o}f's} Type Theory}}.
\newblock \bibinfo{publisher}{Oxford University Press}.

\bibitemdeclare{inproceedings}{PittsAM:aximct}
\bibitem{PittsAM:aximct}
\bibinfo{author}{I.~\surnamestart Orton\surnameend} \& \bibinfo{author}{A.~M.
  \surnamestart Pitts\surnameend} (\bibinfo{year}{2016}):
  \emph{\bibinfo{title}{Axioms for Modelling Cubical Type Theory in a Topos}}.
\newblock In \bibinfo{editor}{J.-M. \surnamestart Talbot\surnameend} \&
  \bibinfo{editor}{L.~\surnamestart Regnier\surnameend}, editors: {\sl
  \bibinfo{booktitle}{25th EACSL Annual Conference on Computer Science Logic
  ({CSL} 2016)}}, {\sl \bibinfo{series}{Leibniz International Proceedings in
  Informatics (LIPIcs)}}~\bibinfo{volume}{62}, \bibinfo{publisher}{Schloss
  Dagstuhl--Leibniz-Zentrum f\"ur Informatik}, \bibinfo{address}{Dagstuhl,
  Germany}, pp. \bibinfo{pages}{24:1--24:19}, \doi{10.4230/LIPIcs.CSL.2016.24}.

\bibitemdeclare{misc}{agdacode}
\bibitem{agdacode}
\bibinfo{author}{A.~M. \surnamestart Pitts\surnameend} \&
  \bibinfo{author}{S.~C. \surnamestart Steenkamp\surnameend}
  (\bibinfo{year}{2021}): \emph{\bibinfo{title}{{Agda} code accompanying this
  paper}}, \doi{10.17863/CAM.73911}.

\bibitemdeclare{article}{RobertsDM:weacpw}
\bibitem{RobertsDM:weacpw}
\bibinfo{author}{D.~M. \surnamestart Roberts\surnameend}
  (\bibinfo{year}{2015}): \emph{\bibinfo{title}{The Weak Choice Principle
  {WISC} may Fail in the Category of Sets}}.
\newblock {\sl \bibinfo{journal}{Studia Logica}} \bibinfo{volume}{103}, pp.
  \bibinfo{pages}{1005--1017}, \doi{10.1007/s11225-015-9603-6}.

\bibitemdeclare{unpublished}{ShulmanM:surcpo}
\bibitem{ShulmanM:surcpo}
\bibinfo{author}{M.~\surnamestart Shulman\surnameend} (\bibinfo{year}{2014}):
  \emph{\bibinfo{title}{The surreals contain the plump ordinals}}.
\newblock
  \urlprefix\url{https://homotopytypetheory.org/2014/02/22/surreals-plump-ordinals/}.
\newblock \bibinfo{note}{Homotopy Type Theory blog}.

\bibitemdeclare{inproceedings}{SprengerC:strirc}
\bibitem{SprengerC:strirc}
\bibinfo{author}{C.~\surnamestart Sprenger\surnameend} \&
  \bibinfo{author}{M.~\surnamestart Dam\surnameend} (\bibinfo{year}{2003}):
  \emph{\bibinfo{title}{On the Structure of Inductive Reasoning: Circular and
  Tree-Shaped Proofs in the $\mu$Calculus}}.
\newblock In \bibinfo{editor}{A.~D. \surnamestart Gordon\surnameend}, editor:
  {\sl \bibinfo{booktitle}{Foundations of Software Science and Computation
  Structures (FoSSaCS 2003)}}, {\sl \bibinfo{series}{Lecture Notes in Computer
  Science}} \bibinfo{volume}{2620}, \bibinfo{publisher}{Springer, Berlin,
  Heidelberg.}, pp. \bibinfo{pages}{425--440},
  \doi{10.1007/3-540-36576-1{\_}27}.

\bibitemdeclare{unpublished}{StreicherT:reamc}
\bibitem{StreicherT:reamc}
\bibinfo{author}{T.~\surnamestart Streicher\surnameend} (\bibinfo{year}{2005}):
  \emph{\bibinfo{title}{Realizability Models for {CZF} + $\neg$ Pow}}.
\newblock
  \urlprefix\url{http://www2.mathematik.tu-darmstadt.de/~streicher/CIZF/rmczfnp.pdf}.
\newblock \bibinfo{note}{Unpublished note}.

\bibitemdeclare{incollection}{StreicherT:unit}
\bibitem{StreicherT:unit}
\bibinfo{author}{T.~\surnamestart Streicher\surnameend} (\bibinfo{year}{2005}):
  \emph{\bibinfo{title}{Universes in Toposes}}.
\newblock In \bibinfo{editor}{L.~\surnamestart Crosilla\surnameend} \&
  \bibinfo{editor}{P.~\surnamestart Schuster\surnameend}, editors: {\sl
  \bibinfo{booktitle}{From Sets and Types to Topology and Analysis, Towards
  Practicable Foundations for Constructive Mathematics}},
  chapter~\bibinfo{chapter}{4}, {\sl \bibinfo{series}{Oxford Logic
  Guides}}~\bibinfo{volume}{48}, \bibinfo{publisher}{Oxford University Press},
  pp. \bibinfo{pages}{78--90}, \doi{10.1093/acprof:oso/9780198566519.001.0001}.

\bibitemdeclare{article}{swan2018wtypes}
\bibitem{swan2018wtypes}
\bibinfo{author}{A.~\surnamestart Swan\surnameend} (\bibinfo{year}{2018}):
  \emph{\bibinfo{title}{{W}-Types with Reductions and the Small Object
  Argument}}.
\newblock {\sl \bibinfo{journal}{ArXiv e-prints}}
  \bibinfo{volume}{arXiv:1802.07588 [math.CT]}.
\newblock \urlprefix\url{https://arxiv.org/abs/1802.07588}.

\bibitemdeclare{article}{TaylorP:intso}
\bibitem{TaylorP:intso}
\bibinfo{author}{P.~\surnamestart Taylor\surnameend} (\bibinfo{year}{1996}):
  \emph{\bibinfo{title}{Intuitionistic Sets and Ordinals}}.
\newblock {\sl \bibinfo{journal}{Journal of Symbolic Logic}}
  \bibinfo{volume}{61}, pp. \bibinfo{pages}{705--744}, \doi{10.2307/2275781}.

\bibitemdeclare{book}{TaylorP:prafm}
\bibitem{TaylorP:prafm}
\bibinfo{author}{P.~\surnamestart Taylor\surnameend} (\bibinfo{year}{1999}):
  \emph{\bibinfo{title}{Practical Foundations of Mathematics}}.
\newblock {\sl \bibinfo{series}{Cambridge Studies in Advanced
  Mathematics}}~\bibinfo{volume}{59}, \bibinfo{publisher}{Cambridge University
  Press}.

\bibitemdeclare{book}{HoTT}
\bibitem{HoTT}
\bibinfo{author}{The \surnamestart {Univalent Foundations Program}\surnameend}
  (\bibinfo{year}{2013}): \emph{\bibinfo{title}{Homotopy Type Theory: Univalent
  Foundations for Mathematics}}.
\newblock \bibinfo{publisher}{Institute for Advanced Study}.
\newblock \urlprefix\url{http://homotopytypetheory.org/book}.

\bibitemdeclare{article}{vandenberg2014axiom}
\bibitem{vandenberg2014axiom}
\bibinfo{author}{B.~\surnamestart {van den Berg}\surnameend} \&
  \bibinfo{author}{I.~\surnamestart Moerdijk\surnameend}
  (\bibinfo{year}{2014}): \emph{\bibinfo{title}{The Axiom of Multiple Choice
  and Models for Constructive Set Theory}}.
\newblock {\sl \bibinfo{journal}{Journal of Mathematical Logic}}
  \bibinfo{volume}{14}(\bibinfo{number}{01}), p. \bibinfo{pages}{1450005},
  \doi{10.1142/S0219061314500056}.

\end{thebibliography}
